\documentclass[a4paper]{article}

\usepackage[utf8]{inputenc}
\usepackage{lmodern}
\usepackage[T1]{fontenc}

\usepackage{color}

\usepackage{mathtools, amssymb, amsthm}
\usepackage{enumitem}
\usepackage{bbm,dsfont}
\numberwithin{equation}{section}

\newcommand{\MR}{\textit{MR}}
\newcommand{\CP}{\!\textit{CP}}
\newcommand{\NLCP}{\!\textit{NLCP}}
\newcommand{\const}{\textit{const}\,}

\newcommand{\K}{\mathds{K}}
\newcommand{\R}{\mathds{R}}
\newcommand{\C}{\mathds{C}}
\newcommand{\N}{\mathds{N}}
\newcommand{\A}{\mathcal{A}}

\newcommand{\D}{\mathcal{D}}

\renewcommand{\L}{\mathcal{L}}
\renewcommand{\d}{\mathrm{d}}
\renewcommand{\Re}{\operatorname{Re}}

\newcommand{\fra}{\mathfrak{a}}

\renewcommand{\mid}{\, \vert \,}

\DeclarePairedDelimiter\abs{\lvert}{\rvert}
 \DeclarePairedDelimiter\norm{\lVert}{\rVert} 

\theoremstyle{plain}
\newtheorem{theorem}{Theorem}[section]
\newtheorem{proposition}[theorem]{Proposition}
\newtheorem{lemma}[theorem]{Lemma}
\newtheorem{corollary}[theorem]{Corollary}

\theoremstyle{definition}

\newtheorem{problem}[theorem]{Problem}

\theoremstyle{remark}
\newtheorem{remark}[theorem]{Remark}

\begin{document}
\title{Maximal Regularity for Evolution Equations Governed by Non-Autonomous Forms}
\author{
    Wolfgang Arendt,
    Dominik Dier,
    Hafida Laasri,
    El Maati Ouhabaz\footnote{Corresponding author.}
}


\maketitle

\begin{abstract}\label{abstract}
We consider a non-autonomous evolutionary problem
\[
\dot{u} (t)+\A(t)u(t)=f(t), \quad u(0)=u_0
\]
where the operator $\A(t)\colon V\to V^\prime$ is associated with a form $\fra(t,.,.)\colon V\times V \to \R$ and $u_0\in V$.
Our main concern is to prove well-posedness with maximal regularity which means the following. 
Given a Hilbert space $H$ such that $V$ is continuously and densely embedded into $H$ and given $f\in L^2(0,T;H)$ we are interested in solutions $u \in H^1(0,T;H)\cap L^2(0,T;V)$. 
We do prove well-posedness in this sense whenever the form is piecewise Lipschitz-continuous and satisfies the square root property.
Moreover, we show that each solution is in $C([0,T];V)$.
The results are applied to non-autonomous Robin-boundary conditions and maximal regularity is used to solve a quasilinear problem.
\end{abstract}

\medskip


\section{Introduction}\label{section:introduction}

The aim of this article is to study non-autonomous evolution equations governed by forms. We consider Hilbert spaces $H$ and $V$ such that $V$ is continuously embedded into $H$ and a form
\[
	\fra\colon [0,T]\times V\times V \to \C
\]
such that $\fra(t, .,.)$ is sesquilinear for all $t\in[0,T]$,
$\fra(.,u,v)\colon [0,T]\to\C$ is measurable for all $u,v \in V$, 
\[
	\abs{ \fra(t,u,v) } \le M \norm{u}_V \norm{v}_V \quad (t\in[0,T]) \qquad \text{  ({\it $V$-boundedness})}
\] 
 and such that 
\[
	\Re \fra (t,u,u) \ge \alpha \|u\|^2_V + \omega \norm{u}_H^2 \quad (u\in V, t\in [0,T]) \qquad \text{  ({\it quasi-coerciveness}) }
\] 
where $M \ge 0$, $\alpha > 0$ and $\omega \in \R$. 
For fixed $t\in [0,T]$ there is a unique operator $\A(t)\in \L(V,V^\prime)$ such that $\fra(t,u,v)= \langle \A(t)u,v\rangle$ for all $u,v \in V$.
\begin{theorem}[Lions' Theorem]\label{thm:Lions}
	Given $f\in L^2(0,T;V^\prime)$ and $u_0\in H$, there is a unique solution $u \in \MR(V,V'):= L^2(0,T;V)\cap H^1(0,T;V')$ of
	the Cauchy problem
	\begin{equation} \label{equ1.1}
	\dot u (t)+\A(t)u(t)=f(t)\quad (t\in (0,T)), \qquad u (0)=u_0.
\end{equation}
\end{theorem}
Note that 
$\MR(V,V^\prime) \hookrightarrow C([0,T], H)$
so  that the initial condition makes sense.
We refer to \cite[p.\ 112]{Sho97}, \cite[XVIII Chapter 3, p.\ 513]{DL92} for the proof of this theorem. 
Lions' theorem states well-posedness of the Cauchy problem \eqref{equ1.1} with maximal regularity in $V'$.
However, the result is not really satisfying since in concrete situations one is interested in solutions which take values in $H$ and not merely in $V^\prime$ (note that $H\hookrightarrow V^\prime$ by the canonical identification). In fact, if we consider boundary value problems, only the part $A(t)$ of $\A(t)$ in $H$ does really realize the boundary conditions in question. So the general problem is whether maximal regularity in $H$ is valid in the following sense:

\begin{problem}\label{Pr1}
If $f\in L^2(0,T; H)$ and $u_0\in V$,  is the solution of (\ref{equ1.1}) in $\MR(V,H):= H^1(0,T;H) \cap L^2(0,T;V)$?
\end{problem}

One has to distinguish the two cases $u_0 = 0$ and $u_0 \neq 0$. For $u_0 = 0$ Problem~\ref{Pr1} is explicitly asked by Lions \cite[p.\ 68]{Lio61}
and seems to be open up to today. A positive answer is given by Lions if $\fra$ is symmetric (i.e.\  $\fra(t,u,v)=\overline{\fra(t,v,u)}$) and $\fra(.,u,v) \in C^1 [0,T]$ for all $u,v \in V$.  
By a completely different
approach a positive answer is also given in \cite{OS10} for general (possibly non-symmetric) forms such that 
$\fra(.,u,v) \in C^\alpha[0,T]$ for all $u,v \in V$ and some
$\alpha > \frac 1 2$. Again, the result in  \cite{OS10} concerns the case $u_0=0$.

Concerning $u_0 \not= 0$ it seems natural to assume $u_0 \in V$ as we did in Problem~\ref{Pr1}. 
However, already in the autonomous case, i.e.\ $A(t) \equiv A$, the solution is in $\MR(V,H)$ if and only if $u_0 \in D(A^{1/2})$,
and it may happen that $V\not\subset D(A^{1/2})$ (failure of the square root property).
So one has to impose a stronger condition on the initial value $u_0$ or the form (e.g.\ symmetry or more generally, the square root property).
Lions \cite[p.\ 94]{Lio61} gave a positive answer for $u_0 \in D(A(0))$ provided that 
$\fra(., u, v) \in C^2[0,T]$ for all $u, v \in V$ and $f \in H^1(0,T;H)$.
Moreover, 
a little bit hidden in his book one finds the following result:
a combination of \cite[Theorem~1.1, p.~129]{Lio61} and \cite[Theorem~5.1, p.~138]{Lio61} shows that 
if $\fra(., u, v) \in C^1[0,T]$ and $\fra(t, u, v)=\overline{\fra(t, v, u)}$ for all
$u, v \in V$, $t\in [0,T]$, then Problem~\ref{Pr1} has a positive answer. 
Finally, we mention a result of Bardos \cite{Bar71} who gives a positive answer to Problem \ref{Pr1} under the assumptions that the domains of both $A(t)^{1/2}$ and 
$A(t)^{*1/2}$ coincide with $V$ as spaces and topologically with constants independent of $t$,
and that $\A(.)^{1/2}$ is continuously differentiable with values in $\L(V,V')$.

Now we explain our contribution to the problem of maximal regularity formulated in Problem~\ref{Pr1}.
We suppose that the sesquilinear form $\fra$ can be written as 
$\fra(t,u,v) = \fra_1(t,u,v) + \fra_2(t,u,v)$
where $\fra_1$ is symmetric, $V$-bounded and coercive as above and piecewise Lipschitz-continuous on $[0,T]$,
whereas $\fra_2\colon [0,T] \times V \times H \to \C$ satisfies $\abs{\fra_2(t,u,v)}\le M_2 \norm u_V \norm v_H$ 
and $\fra_2(.,u,v)$ is measurable for all $u\in V$, $v \in H$.
Furthermore we consider a more general Cauchy problem than $\eqref{equ1.1}$ introducing a multiplicative perturbation
$B\colon [0,T] \to \L(H)$ which is strongly measurable such that $\norm{B(t)}_{\L(H)} \le \beta_1$ for all $t \in [0,T]$ and
$0 < \beta_0 \le (B(t)g \mid g)_H$ for $g \in H$, $\norm{g}_H=1$, $t \in [0,T]$
and study the problem
\begin{equation} \label{eq:CP_with_B}
	\dot u (t)+ B(t)A(t)u(t)=f(t)\quad (t\in (0,T)), \qquad u (0)=u_0
\end{equation}
(where $A(t)$ is the part of $\A(t)$ in $H$).
The multiplicative perturbation is needed for several applications to non-linear problems (see below).
Our main result on maximal regularity is the following (Corollary~\ref{corW}): 
Given $f\in L^2(0,T;H)$, $u_0\in V$ there is a unique solution $u\in H^1(0,T;H) \cap L^2(0,T;V)$ of $\eqref{eq:CP_with_B}$.
This extends the result of Lions mentioned above.

In the case where $B(t) \equiv I$ (or even $B(t)=\beta(t) I$) then we show that Problem~\ref{Pr1} has a positive answer even if $\fra_1$ is not symmetric.
What is needed is the square root property, similar to the assumptions made by Bardos.
Thus we also generalize Bardos' result with a completely different method though.
In fact, the method of this paper is based on a careful analysis of $\A(t)^{1/2}$ which allows us to establish a non-autonomous similarity transform from 
\[
	\MR_\fra(H):=\{u\in H^1(0,T;H)\cap L^2(0,T;V) : \A(.)u(.) \in L^2(0,T;H)\}
\]
to $\MR(V,V')$ (cf.\ Theorem~\ref{thm:iso}).

One of our other results, established in Section 4, shows that the solution is automatically in $C([0,T],V)$. 
In fact, the classical result of Lions says that
\begin{equation}\label{eq:embedding_in_continuous_functions}
    \MR(V,V') \hookrightarrow C([0,T];H),
\end{equation}
see \cite[p.\ 106]{Sho97}. In the non-autonomous situation considered here we prove that $\MR_\fra(H)$
is continuously embedded into $C([0,T]; V)$.

Note that if $u \in \MR(V,H)$ is a solution of $\eqref{equ1.1}$,  then automatically $u \in \MR_\fra(H) \subset C([0,T]; V)$. 
It is this continuity with values in $V$ which allows us to weaken the regularity assumption on the form $\fra(t, ., .)$ from 
Lipschitz-continuity in Theorem~\ref{thm:well-posedness_in_H} to piecewise Lipschitz continuity on $[0,T]$ in Corollary~\ref{corW}.

The embedding of $\MR_\fra(H)$ into $C([0,T]; V)$ has important consequences for applications (see e.g.\ \cite{ADK13}).
Moreover, we also show that the embedding of $\MR_\fra(H)$ into $L^2(0,T;V)$ is compact whenever $V$ is compactly embedded in $H$.
This is important for our application to quasilinear problems given in Section~\ref{sec:quasilinear}.

We illustrate our abstract results by three applications.  
One of them concerns the heat equation with non-autonomous Robin-boundary-conditions
\begin{equation} \label{equ1.2}
\partial_\nu u(t)+\beta(t)u(t) \vert_{\partial\Omega}=0
\end{equation} 
on a bounded Lipschitz domain $\Omega$.
Here $\partial_\nu$ denotes the normal derivative.
Under appropriate assumptions on $\beta$ we prove {\it maximal regularity}, i.e., that the solution is in $\MR(H^1(\Omega),L^2(\Omega))$.  
This is of great importance if non-linear problems are considered. 
As an example we prove existence of a solution of the problem
\begin{equation*}
 \left\{  \begin{aligned}
          \dot u(t)  &=  m(t, x, u(t), \nabla u(t)) \Delta u(t) +  f(t)\\
                        u(0)    &= u_0 \in H^1(\Omega)\\
                         \partial_\nu u(t)&+ \beta (t, .)u (t )= 0 \mbox{  on  } \partial \Omega\\
         \end{aligned} \right.
\end{equation*}
i.e., a quasilinear problem with non-autonomous Robin boundary conditions on a bounded Lipschitz domain $\Omega \subset \R^d$. 
It is here that we need well-posedness and maximal regularity of problem $\eqref{eq:CP_with_B}$ with multiplicative perturbation
(of the form $Bg= m(t, x, u(t), \nabla u(t)) g$).
Previous results (see \cite{AC10}) did not allow non-autonomous boundary conditions. Finally, we also mention that our main result, Corollary
\ref{corW}, with a suitable $B(t)$ is used in an essential way in \cite{ADK13} to prove a well-posedness  result  for an evolutionary problem
on a network with time dependent transmission conditions. 

The paper is organized as follows.  Section 2 has preliminary character and is devoted to estimates on operators associated with forms. 
In Section 3 we prove that multiplication by $\A^{1/2}(.)$ defines an isomorphism from $\MR_\fra(H)$ onto $\MR(V,V')$.
This result is our main tool for the remainder of the paper.
In the same section it is shown that $\MR_\fra(H)$ is continuously embedded into $C([0,T];V)$ 
and also compactly embedded into $L^2([0,T]; V)$ if in addition the embedding of $V$ in $H$ is compact. 
Our main result on well-posedness with maximal regularity in $H$ is established in Section 4.
A  series of examples concerning parabolic equations is given in Section 5, where the main point concerns non-autonomous boundary conditions. 
Several new mapping theorems for vector-valued one-dimensional (mixed) Sobolev spaces are proved in the appendix.

\subsection*{Acknowledgment}
It is a pleasure to thank Marjeta Kramar for stimulating discussions on non-autonomous boundary value problems.
The authors obtained diverse financial support which they gratefully acknowledge: 
D.\ Dier is a member of the DFG Graduate School 1100: Modelling, Analysis and Simulation in Economathematics, 
H.\ Laasri stayed at the University of Ulm with the help of a DAAD-grant, 
E.\ M.\ Ouhabaz visited the University of Ulm in the framework of the Graduate School: 
Mathematical Analysis of Evolution, Information and Complexity financed by the Land Baden-Württemberg and 
W.\ Arendt enjoyed a wonderful research stay at the University of Bordeaux.
The research of E.\ M.\ Ouhabaz  is partly supported by the ANR project  ``Harmonic Analysis at its Boundaries'',  ANR-12-BS01-0013-02.

\section{Forms and associated operators}\label{section:forms}

Throughout this paper the underlying field is $\K = \C$ or $\R$.
This means that all results are valid no matter whether the
underlying field is $\R$ or $\C$. Let $V, H$ be two Hilbert spaces
over $\K$.  Their   scalar  products and the corresponding norms will be denoted by  $(. \mid .)_H$, $(. \mid .)_V$, $\norm{.}_H$ and $\norm{.}_V$, respectively.  We assume that
\[
    V \underset d \hookrightarrow H;
\]
i.e., $V$ is a dense subspace of $H$ such that for some constant
$c_H >0$,
\begin{equation}\label{eq:V_dense_in_H}
    \norm{u}_H \le c_H \norm u _V \quad (u \in V).
\end{equation}
Let
\[
    \fra\colon V \times V \to \K
\]
be sesquilinear and \emph{continuous}; i.e.\
\begin{equation}\label{eq:a_continuous}
    \abs{\fra(u,v)} \le M \norm u _V \norm v _V \quad (u,v \in V)
\end{equation}
for some constant M. We assume that $\fra$ is \emph{quasi-coercive};
i.e.\ there exist constants $\alpha >0$, $\omega \in \R$ such that
\begin{equation}\label{eq:H-elliptic}
    \Re \fra(u,u) + \omega \norm u_H^2 \ge \alpha \norm u _V^2 \quad (u \in V).
\end{equation}
If $\omega=0$, we say that  the form $\fra$ is  \emph{coercive}. The
operator $\A \in \L(V,V')$ associated with $\fra$ is defined by
\[
\langle \A u, v \rangle = \fra(u,v) \quad (u,v \in V). 
\]
Here  $V'$ denotes  the antidual of $V$ when  $\K =\C$ and the dual when $\K =\R$. The duality 
between $V'$ and $V$ is denoted by $\langle ., . \rangle$. 
As usual, we identify $H$ with a dense subspace of $V'$ (associating to $f \in H$ the antilinear form $v \mapsto ( f \mid v )_H$).
This embedding is continuous, in fact
\begin{equation*}
	\norm{f}_{V'} \le c_H \norm{f}_H \quad (u \in H),
\end{equation*}
with the same constant $c_H$ as in \eqref{eq:V_dense_in_H}.

Seen as an unbounded operator on $V'$ with domain $D(\A) = V$ the
operator $- \A$ generates a holomorphic semigroup on $V'$. In the
case where $\K=\R$ we mean by this that the $\C$-linear extension
of $-\A$ on the complexification of $V'$ generates a holomorphic
$C_0$-semigroup. The semigroup is bounded on a  sector if $\omega
=0$, in which case $\A$ is an isomorphism. Denote by $A$ the part
of $\A$ on $H$; i.e.,\
\begin{align*}
    D(A) := {}& \{ u\in V : \A u \in H \}\\
    A u = {}& \A u.
\end{align*}
Then $-A$ generates a holomorphic $C_0$-semigroup on $H$ (the
restriction of the semigroup generated by $-\A$ to $H$). For all this, see e.g.\ the monographs 
\cite[Chap.\ 1]{Ouh05}, \cite[Chap.\ 2]{Tan79}. 

For the remainder of this section we assume that $\fra \colon V \times V \to \K$ is a sesquilinear form satisfying \eqref{eq:V_dense_in_H}, \eqref{eq:a_continuous} and \eqref{eq:H-elliptic} with $\omega=0$ and $\A$ is the associated operator of $\fra$.
We show some estimates for the resolvent of the operator $\A$.
\begin{proposition}\label{prop:bounds}
    For $\lambda \ge 0$ we have
    \begin{enumerate}
	\item[{\rm a)}] $\norm{(\lambda + \A)^{-1}}_{\L(V',V)} \le 1/\alpha$,
        \item[{\rm b)}] $\norm{(\lambda + \A)^{-1}}_{\L(V',H)} \le (\frac\alpha {2 c_H} + \sqrt{\lambda} \sqrt{2 \alpha })^{-1}$,
        \item[{\rm c)}] $\norm{(\lambda + \A)^{-1}}_{\L(H,V)} \le (\frac\alpha {2 c_H} + \sqrt{\lambda} \sqrt{2 \alpha })^{-1}$,
        \item[{\rm d)}] $\norm{(\lambda + \A)^{-1}}_{\L(H)} \le (\frac\alpha{c_H^2} + \lambda )^{-1}$,
        \item[{\rm e)}] $\norm{(\lambda + \A)^{-1}}_{\L(V')} \le (\frac\alpha {2c_H^2} + \lambda \frac \alpha{2(\alpha+M)})^{-1}$, and finally
        \item[{\rm f)}] 	$\norm{(\lambda + \A)^{-1}}_{\L(V)} \le \frac M \alpha(\frac\alpha {2c_H^2} + \lambda \frac \alpha{2(\alpha+M)})^{-1}$.
    \end{enumerate}
\end{proposition}
\begin{proof}
Let $u \in V$ and $\lambda \ge 0$, then
\begin{equation}\label{eq:c_est}
	\langle (\A+\lambda) u, u \rangle = \fra(u,u) + \lambda \norm{u}_H^2 \ge \alpha \norm{u}_V^2 + \lambda \norm{u}_H^2.
\end{equation}

By \eqref{eq:c_est} we have
\[
	\norm{(\A+\lambda)u}_{V'} \norm{u}_V \ge \alpha \norm{u}^2_V.
\]
Dividing by $\norm{u}_V$ shows a).

By \eqref{eq:c_est} and the inequality $a^2+b^2 \ge 2ab$, $a,b\in\R$ we have
\begin{align*}
	\norm{(\A+\lambda)u}_{V'} \norm{u}_V &\ge \alpha \norm{u}^2_V+ \lambda \norm{u}_H^2\\
		&\ge \tfrac \alpha 2 \norm{u}_V^2 + 2\sqrt{\tfrac \alpha 2}\norm{u}_V \sqrt{\lambda}\norm{u}_H\\
		&\ge \tfrac \alpha {2 c_H} \norm{u}_V \norm{u}_H + \sqrt{\lambda} \sqrt{2 \alpha }\norm{u}_V \norm{u}_H.
\end{align*}
Dividing by $\norm{u}_V$ shows b).

If $u \in D(A)$ we obtain similarly 
\[
	\norm{(\A+\lambda)u}_{H} \norm{u}_H
		\ge \tfrac \alpha {2 c_H} \norm{u}_V \norm{u}_H + \sqrt{\lambda} \sqrt{2 \alpha }\norm{u}_V \norm{u}_H.
\]
Dividing by $\norm{u}_H$ shows c).

For $u \in D(A)$ we obtain by \eqref{eq:c_est} that
\begin{align*}
	\norm{(\A+\lambda)u}_{H} \norm{u}_H &\ge \alpha \norm{u}^2_V+ \lambda \norm{u}_H^2\\
		&\ge \frac\alpha{c_H^2} \norm{u}^2_H+ \lambda \norm{u}_H^2.
\end{align*}
Dividing by $\norm{u}_H$ shows d).

Let $u,v \in V$, $\norm{v}_V=1$. For $0<\mu<1$ we set
\[
	w:= \mu \frac{u}{\norm{u}_V}+ (1-\mu)v,
\]
then $\norm{w}_V \le \mu + (1-\mu) = 1$. Thus
\begin{align*}
	\norm{(\A+\lambda)u}_{V'}  &\ge \Re \fra(u,w) + \Re( \lambda (u \mid w)_H )\\
		&=\frac \mu{\norm{u}_V} \Re\fra(u,u) + (1-\mu)\Re \fra(u,v)\\
			&\quad + \lambda \frac \mu{\norm{u}_V} \norm{u}_H^2 + \lambda (1-\mu)  \Re(u \mid v)_H\\
		&\ge  \mu \alpha \norm{u}_V - (1-\mu)M\norm{u}_V + \lambda (1-\mu)  \Re(u \mid v)_H.
\end{align*}
If we choose $\mu =(\frac{\alpha}2 +M) / (\alpha+M)$ and take the supremum over $v\in V$ with $\norm{v}_V=1$, we obtain
\begin{align*}
	\norm{(\A+\lambda)u}_{V'} &\ge \frac\alpha 2 \norm{u}_V + \lambda \frac \alpha{2(\alpha+M)}  \norm{u}_{V'}\\
		&\ge \frac\alpha {2c_H^2} \norm{u}_{V'} + \lambda \frac \alpha{2(\alpha+M)}  \norm{u}_{V'}.
\end{align*}
This proves e).

Finally 
\[
	\norm{(\lambda+\A)^{-1}}_{\L(V)} \le \norm{\A^{-1}}_{\L(V',V)}  \norm{(\lambda+\A)^{-1}}_{\L(V')} \norm{\A}_{\L(V,V')}.
\]
This shows f).
\end{proof}

Next we define the operator $\A^{-1/2} \in \L(V')$ by
\[
    \A^{-1/2} \varphi := \frac 1 \pi \int_0^\infty \lambda^{-1/2} (\lambda+\A)^{-1} \varphi \ \d \lambda \quad (\varphi\in V'),
\]
see \cite[(3.52)]{ABHN11} or \cite[Sec.\ 2.6 p.\ 69]{Paz83}.
Then $(\A^{-1/2})^2 =  \A^{-1}$.
Moreover, $\A^{-1/2}$ is injective.
One defines $\A^{1/2}$ by $D(\A^{1/2}) = R(\A^{-1/2})$, $\A^{1/2}= (\A^{-1/2})^{-1}$.
Then $-\A^{1/2}$ is a closed operator on $V'$ (in fact, the generator of a holomorphic $C_0$-semigroup).
Denoting as before the part of $\A$ in $H$ by $A$.
Then $A$ is invertible and $A^{-1/2}f = \A^{-1/2}f$ for $f \in H$.
Then $A^{-1/2}$ is injective and $D(A^{1/2}) = R(A^{-1/2})$, $A^{1/2}= (A^{-1/2})^{-1}$.
It can happen that $R(A^{-1/2}) \neq V$.
The following is easy to see using that $(\A^{-1/2})^2 = \A^{-1}$ is an isomorphism from $V'$ onto $V$.
\begin{proposition}
	The following are equivalent
	\begin{enumerate}[]
		\item[{\rm(i)}] $R(A^{-1/2})=V$,
		\item[{\rm(ii)}] $R(\A^{-1/2})=H$.
	\end{enumerate}
\end{proposition}
We say that the form $\fra$ has the \emph{square root property} if these two equivalent conditions are satisfied.
In that case $\A^{-1/2}$ is an isomorphism from $V'$ onto $H$ with inverse $\A^{1/2}$
and $A^{-1/2}$ is an isomorphism from $H$ onto $V$ with inverse $A^{1/2}$.
Moreover, $A^{1/2}$ is the part of $\A^{1/2}$ in $H$.

\begin{remark}
	It is known that the square root property is equivalent to $D(A^{1/2*})=V$.
	Thus if the form $\fra$ satisfies the square root property it is given by 
	\[
		\fra(u,v) = (A^{1/2}u \mid  A^{1/2*}v)_H \quad (u,v \in V).
	\]
\end{remark}

If $\fra$ is symmetric, or more generally if $\fra = \fra_1 + \fra_2$, 
where $\fra_1\colon V\times V \to \K$ is symmetric, continuous and coercive and $\fra_2 \colon V \times H \to \K$ is continuous,
then $\fra$ satisfies the square root property.
More generally if $\{\fra(u) : u \in V, \norm{u}_H = 1\}$ lies in a parabola, then the square root property is satisfied (c.f.\ \cite{McI82}).
Finally if $\Omega \subset \R^d$ is a Lipschitz domain and $V=H^1(\Omega)$ or $H^1_0(\Omega)$ and 
\[
	\fra(u,v) = \int_\Omega \sum_{j,k} a_{jk} \partial_j u \overline{\partial_k v} \ \d{x}
\]
where $a_{jk} \in L^\infty(\Omega)$ are real coefficients satisfying
\[
	\sum_{j,k} a_{jk} \xi_j \xi_k \ge \alpha \abs\xi^2 \quad (\xi \in \R^d)
\]
a.e.\ where $\alpha >0$, then $\fra$ has the square root property.
This is a version of the solution of the famous Kato square root problem (see \cite{AT03}).

The square root property implies that there exists a constant $\gamma >0$ such that
\begin{equation}\label{eq:sqrt_est_H}
	\gamma \norm{u}_V \le \norm{A^{1/2}u}_H \le \frac 1 \gamma \norm{u}_V
\end{equation}
for all $u \in V$. As a consequence
\begin{equation}\label{eq:sqrt_est_V'}
	\alpha \gamma \norm{f}_H \le \norm{\A^{1/2}f}_{V'} \le \frac M \gamma \norm{f}_H \quad (f \in H).
\end{equation}
\begin{proof}
	Let $f \in H$. Then $\A^{1/2}f \in V'$ and 
	\[
		\norm{\A^{1/2}f}_{V'} = \norm{\A\A^{-1/2}f}_{V'}\le M \norm{A^{-1/2}f}_{V} \le \frac M \gamma \norm{f}_{H}
	\]
	by \eqref{eq:sqrt_est_H} which shows the second inequality.
	Moreover
	\begin{multline*}
		\norm{f}_{H} \le \frac 1 \gamma \norm{A^{-1/2}f}_{V} = \frac 1 \gamma \norm{\A^{-1}\A A^{-1/2}f}_{V}\\
			\le \frac 1 {\alpha\gamma} \norm{\A A^{-1/2}f}_{V'} = \frac 1 {\alpha\gamma} \norm{\A^{1/2}f}_{V'},
	\end{multline*}
	which is the first inequality.
\end{proof}
		
	The constants in Proposition~\ref{prop:bounds} only depend on the continuity constant $M$, the coerciveness constant $\alpha$ of the form $\fra$
	and the norm $c_H$ of the embedding of $V$ into $H$.
	This is in general not true for the constant $\gamma$ in \eqref{eq:sqrt_est_H}. And indeed there are forms which do not have the square root property.
	However, if we ask for further properties of the form $\fra$, a universal constant $\gamma >0$ can be found.
	For example, if $\fra$ is symmetric, then actually
	\[
		\sqrt\alpha \norm{u}_V \le \norm{A^{1/2}u}_H \le \sqrt M \norm{u}_V \quad (u\in V).
	\]

\section{An isomorphism for $\MR$ spaces}\label{section:continuous_functions}

In this section we show that multiplication by $\A(.)^{1/2}$  defines an isomorphism from  $\MR_\fra(H)$ onto $\MR(V,V')$  (Theorem~\ref{thm:iso}). This will be our main tool in the next section and has interesting consequences by itself (Corollary~\ref{cor:embedding_in_continuous_functions}).
Let $V,H$ be separable Hilbert spaces
over $\K = \R$ or $\C$ such that $V \underset d \hookrightarrow
H$. Let $T>0$ and 
\[
    \fra\colon [0,T]\times V \times V \to \K
\]
be a function such that $\fra(t,.,.)\colon V \times V\to \K$ is sesquilinear for all $t \in [0,T]$.
We assume that $\fra$ is $V$-bounded, and coercive (i.e. quasi-corecive with $\omega = 0$), see Introduction. 
In addition we suppose that $\fra$ is \emph{Lipschitz continuous}; i.e.,  
\[
    \abs{\fra(t,u,v)- \fra(s,u,v)} \le \dot M \abs{t-s} \norm u_V \norm v_V \quad (t,s \in [0,T], u,v \in V).
\]

\begin{remark}
    It follows from the Uniform Boundedness Principle that $\fra$
    is Lipschitz continuous whenever $\fra(., u, v) \colon [0,T] \to \K$
    is Lipschitz continuous for all $u,v \in V$.
\end{remark}

We denote by $\A(t) \in \L(V,V')$ the operator associated with
$\fra(t,.,.)$. 
Further we assume that the non-autonomous form $\fra$ has the \emph{square root property}, by which we mean the following
\begin{enumerate}
\item[a)] each form $\fra(t,.,.)$ has the square root property $(t\in [0,T])$ and
\item[b)] there exists a constant $\gamma>0$ such that
\[
	\gamma \norm{v}_V \le \norm{\A(t)^{1/2}v}_H \le \frac 1 \gamma \norm{v}_V \quad (t \in [0,T], v\in V).
\]
\end{enumerate}
In the following we let 
\[
	\A^{1/2}(t) = (\A(t))^{1/2},\quad \A^{-1/2}(t) = (\A(t))^{-1/2} \quad (t \in [0,T]).
\]
Thus $\A^{1/2}$ is a mapping from $[0,T]$ into $\L(H;V')$ and $\A^{-1/2}$ from $[0,T]$ into $\L(V',H)$.
We will also consider these mappings with values in different spaces (such as $\L(V,V')$ in the first case) without changing the notation.
We consider the following maximal regularity
space
\[
    \MR_\fra(H) := \{ u \in H^1(0,T;H) \cap L^2(0,T;V) : \A(.)u(.) \in L^2(0,T;H) \}.
\]
It is a Hilbert space for the norm $ \norm . _{\MR_\fra(H)}$ given by 
\[
    \norm u _{\MR_\fra(H)}^2 := \norm u ^2_{ L^2(0,T;V)} + \norm{\dot u} ^2_{ L^2(0,T;H)} + \norm {\A(.)u(.)} ^2_{ L^2(0,T;H)}.
\]
Under the above assumptions on the form $\fra$ our main result of this section says the following.

\begin{theorem}\label{thm:iso}
	The mapping $u \mapsto \A^{1/2} u$
	defines an isomorphism from $\MR_\fra(H)$ onto $\MR(V,V')$.
	Moreover for $u \in \MR_\fra(H)$ 
	\begin{equation}\label{eq:chain_rule_for_Au}
        (\A^{1/2}(.) u(.))\dot{} =  (\A^{1/2})\dot{}(.) u(.) + \A^{1/2}(.) \dot u(.)
    \end{equation}
    in $L^2(0,T;V')$ and for $v \in \MR(V,V')$ 
    \begin{equation}\label{eq:chain_rule_for_v}
        (\A^{-1/2}(.) v(.))\dot{} = (\A^{-1/2})\dot{}(.) v(.) + \A^{-1/2}(.) \dot v(.)
    \end{equation}
    in $L^2(0,T;H)$.
\end{theorem}
The proof depends on several lemmas which we show below.
The first, Lemma~\ref{lem:square_root_lipschitz_continuous} b) shows that
$\A^{1/2} \colon  [0,T]\to \L(V,V')$ is Lipschitz continuous. 
Thus by Proposition \ref{prop:lipschitz_continuous_operators}
b) there exists $(\A^{1/2})\dot{}\, \colon [0,T]\to \L(V,V')$, which is strongly measurable and bounded. 
Thus for $u \in L^2(0,T;V)$, $(\A^{1/2})\dot{} (.) u(.) \in L^2(0,T;V')$, which explains
that the first term on the right hand side of $\eqref{eq:chain_rule_for_Au}$ is well-defined. 
Concerning the second term, we will see in Lemma~\ref{lem:square_root_strongly_continuous} that $\A^{1/2} \colon [0,T]\to \L(H,V')$ 
is strongly measurable and bounded. 
Thus, for $u\in H^1(0,T;H)$ one has $\A^{1/2} (.) \dot u(.) \in
L^2(0,T;V')$. Thus the right hand side of
$\eqref{eq:chain_rule_for_Au}$ is indeed in $L^2(0,T;V')$.
For similar reasons also \eqref{eq:chain_rule_for_v} is well defined.

\begin{corollary}\label{cor:embedding_in_continuous_functions}
    The space $\MR_\fra(H)$ is continuously embedded into $C([0,T];V)$.
    Moreover, if the embedding $V \hookrightarrow H$ is compact, 
    then also the embedding of $\MR_\fra(H)$ into $L^2(0,T;V)$ is compact.
\end{corollary}

For the proof of Theorem~\ref{thm:iso} and Corollary~\ref{cor:embedding_in_continuous_functions} we need several
auxiliary results.

\begin{lemma}\label{lem:square_root_lipschitz_continuous}
    The mappings
    \begin{enumerate}[label={\rm \alph*)}]
        \item $\A^{-1/2}\colon [0,T] \to \L(V)$,
        \item $\A^{1/2}\colon [0,T] \to \L(V,V')$ and
        \item $\A^{-1/2}\colon [0,T] \to \L(H)$
    \end{enumerate}
    are Lipschitz continuous.
\end{lemma}
\begin{proof}
    a) Let $u\in V$. We have 
    \begin{align*}
     \A^{-1/2}(t)u - \A^{-1/2}(s)u 
            &=  \frac 1 \pi \int_0^\infty \lambda^{-1/2} \big[(\lambda+\A(t))^{-1} - (\lambda+\A(s))^{-1} \big] u \ \d \lambda \\
           & = \frac 1 \pi \int_0^\infty \lambda^{-1/2} (\lambda+\A(t))^{-1}( \A(s)-\A(t) ) (\lambda+\A(s))^{-1} u \ \d \lambda.
            \end{align*}
   It follows from   Proposition \ref{prop:bounds} a) and f) that 
            \begin{align*}
        \norm[\big]{\A^{-1/2}&(t)u - \A^{-1/2}(s)u}_V\\
            &\le \frac 1 \alpha \int_0^\infty \lambda^{-1/2} \norm*{ ( \A(s)-\A(t) ) (\lambda+\A(s))^{-1} u }_{V'} \ \d \lambda\\
            &\le \frac{\dot M} \alpha \abs{s-t} \int_0^\infty \lambda^{-1/2} \norm*{ (\lambda+\A(s))^{-1} u }_V \ \d \lambda\\
            &\le \const  \abs{s-t} \int_0^\infty \lambda^{-1/2} (\lambda+1)^{-1} \ \d \lambda \ \norm{ u }_V.
    \end{align*}
    b) Let $u \in V$. Then,
    \begin{align*}
        \norm[\big]{\A^{1/2}&(t)u - \A^{1/2}(s)u}_{V'}\\
            &=\norm[\big]{\A(t) \A^{-1/2}(t)u - \A(s) \A^{-1/2}(s)u}_{V'}\\
            &\le \norm[\big]{( \A(t) - \A(s) ) \A^{-1/2}(t)u}_{V'} + \norm[\big]{\A(s) ( \A^{-1/2}(t)u - \A^{-1/2}(s)u )}_{V'}\\
            &\le \abs{t-s} \dot M \norm[\big]{ \A^{-1/2}(t)u }_V + M \norm[\big]{\A^{-1/2}(t)u - \A^{-1/2}(s)u}_V\\
            &\le \abs{t-s} \const \norm{u}_V, 
    \end{align*}
    where we have used part a) above in the last inequality.
    
    The proof of c) is similar to a) using Proposition \ref{prop:bounds} c) and b) instead of a) and f).
\end{proof}

\begin{lemma}\label{lem:square_root_strongly_continuous}
    The mappings
    \begin{enumerate}[label={\rm \alph*)}]
        \item $\A^{-1/2}\colon [0,T] \to \L(H,V)$,
        \item $\A^{1/2}\colon [0,T] \to \L(H,V')$,
        \item $\A^{-1/2}\colon [0,T] \to \L(V',H)$ and
        \item $\A^{1/2}\colon [0,T] \to \L(V,H)$
    \end{enumerate}
    are strongly continuous.
\end{lemma}
\begin{proof}
    a) By the square root property we have
        \[
            \norm[\big]{\A^{-1/2}(t) u}_V \le \tfrac 1 \gamma \norm u_H \quad (u \in H, t \in [0,T]).
        \]
        Since by Lemma \ref{lem:square_root_lipschitz_continuous} a)
        $\A^{-1/2}(.)u \colon [0,T] \to V$ is continuous for $u \in V$,
        the claim follows by a $3\epsilon$-argument.
            
    b) Let $u \in V$. Since $\A^{1/2}(.)u = \A(.) \A^{-1/2}(.)u$ b) follows by a) and the Lipschitz continuity of $\A\colon [0,T] \to \L(V,V')$.
    
        c) By the square root property and \eqref{eq:sqrt_est_V'} we have
        \[
            \norm[\big]{\A^{-1/2}(t) u}_H \le \tfrac 1 {\alpha\gamma} \norm u_{V'} \quad (u \in V', t \in [0,T]).
        \]
        Since by Lemma \ref{lem:square_root_lipschitz_continuous} c)
        $\A^{-1/2}(.)u \colon [0,T] \to H$ is continuous for $u \in H$,
        the claim follows by a $3\epsilon$-argument.
            
    d) Let $u \in H$. Since $\A^{1/2}(.)u = \A^{-1/2}(.)\A(.)u$ d) follows by c) and the Lipschitz continuity of $\A\colon [0,T] \to \L(V,V')$.
\end{proof}

Next we consider the Hilbert space
\[
    \MR(V,H) := L^2(0,T;V) \cap H^1(0,T;H)
\]
with norm
\[
    \norm{u}_{\MR(V,H)}^2 := \norm u_{L^2(0,T;V)}^2 + \norm{\dot u}_{L^2(0,T;H)}^2.
\]
Similarly, we define the Hilbert space
\[
    \MR(H,V') = L^2(0,T;H) \cap H^1(0,T;V')
\]
with norm
\[
    \norm u_{\MR(H,V')}^2 = \norm u_{L^2(0,T;H)}^2 + \norm{\dot u}_{L^2(0,T;V')}^2.
\]

\begin{proof}[Proof of Theorem \ref{thm:iso}]
	If follows from Lemma~\ref{lem:square_root_strongly_continuous} that $\Phi \colon u(.) \mapsto \A^{1/2}(.) u(.)$ defines an isomorphism from
	$L^2(0,T;H)$ onto $L^2(0,T;V')$. Thus it suffices to show that $\Phi \MR_\fra(H) = \MR(V,V')$.
	
	a) Let $u \in \MR_\fra(H)$. Then $u \in \MR(V,H)$.
	Since $\A^{1/2} \colon [0,T]\to \L(H,V')$ is strongly measurable and bounded and Lipschitz continuous with values in $\L(V,V')$, 
	it follows from Proposition~\ref{prop:Lip_MR} that $\A^{1/2}u \in H^1(0,T;V')$. 
	Since by assumption $\A u \in L^2(0,T;H)$ and $\A^{-1/2}\colon [0,T]\to \L(H,V)$ is strongly measurable and bounded, 
	it follows that $\A^{1/2}u = \A^{-1/2} \A u \in L^2(0,T;V)$.
	
	b) Conversely, let $u \in \MR(V,V')$. 
	We have to show that $\A^{-1/2}u \in \MR_\fra(H)$.
	Observe that $u \in \MR(H,V')$.
	Since $\A^{-1/2} \colon [0,T]\to \L(V',H)$ is strongly measurable and bounded and Lipschitz continuous with values in $\L(H)$, 
	it follows from Proposition~\ref{prop:Lip_MR} that $\A^{-1/2}u \in H^1(0,T;H)$. 
	Moreover, $\A \A^{-1/2} u = \A^{1/2}u \in L^2(0,T;H)$ since
	$\A^{1/2} \colon [0,T] \to \L(V,H)$ is strongly measurable and bounded.
\end{proof}

Now we are in the position to  prove Corollary
\ref{cor:embedding_in_continuous_functions}.

\begin{proof}[Proof of Corollary \ref{cor:embedding_in_continuous_functions}]
    Let $u \in \MR_\fra(H)$. Then $\A^{1/2}u \in \MR(V,V')$ by Theorem~\ref{thm:iso}.
    Using this and the classical continuity result $\eqref{eq:embedding_in_continuous_functions}$ it follows  that
    $\A^{1/2}(.)u(.) \in C([0,T];H).$
    Now Lemma \ref{lem:square_root_strongly_continuous} a) implies that
    \[
        u = \A^{-1/2}(.) \A^{1/2}(.) u(.) \in C([0,T];V)
    \]
    which is the first assertion of Corollary \ref{cor:embedding_in_continuous_functions}.
    
    Since by the theorem of Aubin-Lions \cite[p.\ 106]{Sho97} the embedding of $\MR(V,V')$ into $L^2(0,T;H)$ is compact if $V \hookrightarrow H$ is compact,
    it follows from Theorem~\ref{thm:iso} and Lemma~\ref{lem:square_root_strongly_continuous} a) that the embedding of $\MR_\fra(H)$ into $L^2(0,T;V)$ is compact.
\end{proof}

\section{Well-posedness in H}\label{section:well-posedness_in_H}

Let $V,H$ be separable Hilbert spaces such that $V \underset d
\hookrightarrow H$ and let 
\[
    \fra\colon [0,T] \times V \times V \to \K
\]
be a form on which we impose the following conditions. 
It can be written as the sum of two non-autonomous forms
\[
    \fra(t,u,v)= \fra_1(t,u,v)+ \fra_2(t,u,v) \quad (t \in [0,T], u,v \in V)
\]
where
\[
    \fra_1\colon [0,T] \times V \times V \to \K
\]
satisfies the assumptions considered in Section~\ref{section:continuous_functions}; i.e.,
\begin{enumerate}[label={\alph{*})}]
    \item $\abs{\fra_1(t,u,v)} \le M_1 \norm u_V \norm v_V$ for all $u,v \in V$, $t \in [0,T]$;
    \item $\Re \fra_1(t,u,u) \ge \alpha \norm u _V^2$ for all $u \in V$, $t \in [0,T]$ with $\alpha > 0$;
    \item $\fra_1$ satisfies the square root property.
    \item $\fra_1$ is Lipschitz-continuous; i.e., there exists a constant $\dot M_1$
    	\[
			\abs{\fra_1(t,u,v) - \fra_1(s,u,v)} \le \dot M_1 \abs{t-s} \norm u_V \norm v_V
    	\]
    	for all $u,v \in V,$ $s,t \in [0,T]$,
\end{enumerate}
and
\[
    \fra_2\colon [0,T] \times V \times H\to \K
\]
satisfies
\begin{enumerate}[label={(\alph*)}]
    \item[e)] $\abs{\fra_2(t,u,v)} \le M_2 \norm u_V \norm v_H$ for all $u\in V,\  v\in H$, $t \in [0,T]$,
	\item[f)] $\fra_2(.,u,v)\colon [0,T]  \to\K$ is measurable for all $u,v \in V$.
\end{enumerate}
We denote by $\A(t)$ the operator given by $\langle \A(t) u, v \rangle = a(t,u,v)$ and by $\A_1(t)$ the operator given by $\langle \A_1(t) u, v \rangle = a_1(t,u,v)$.
Further we denote by $A(t)$ and $A_1(t)$ the part of $\A(t)$ and $\A_1(t)$ in $H$, respectively.
Finally we set $\A_2(t) := \A(t)-\A_1(t)$. Thus $\langle \A_2(t) u, v \rangle = a_2(t,u,v)$ and by e) and f) $\A_2$ defines a strongly measurable and bounded mapping form $[0,T]$ to $\L(V,H)$.

Let $B\colon[0,T] \to \L(H)$ be a strongly measurable function satisfying
\begin{enumerate}
	\item[g)] $\Re \langle \A^{1/2}_1 B(t) \A^{1/2}_1 u , u \rangle \ge \beta_0 \Re \fra_1(u,u)$ for all $u \in V$, where $\beta_0>0$ and
	\item[h)] $\norm{B(t)}_{\L(H)} \le \beta_1$
\end{enumerate}
for all $t\in [0,T]$.
\begin{remark}
	\begin{enumerate}
		\item[1.)] If $\fra_1$ is symmetric then g) is satisfied if and only if
			$\Re (B(t) g \mid g)_H \ge \beta_0 \norm g_H^2$ for all $g \in H$ and all $t \in [0,T]$.
		\item[2.)] In the general case, $B(t) = \beta(t) I$ with $\beta:[0,T] \to [\beta_0,\beta_1]$ a measurable function and $0<\beta_0<\beta_1$ satisfies both conditions.
\end{enumerate}
\end{remark}

Now we state  our  results  on existence and uniqueness.

\begin{theorem}\label{thm:well-posedness_in_H}
    Let $u_0 \in V$, $f \in L^2(0,T;H)$.
    Then there exists a unique
        $u \in \MR_\fra(H)$
    satisfying
    \begin{equation*}
    \left\{\begin{aligned}
       \dot u(t) +  B(t) \A(t)u(t)  &= f(t) \quad \text{a.e.}\\
                        u(0)    &=u_0.
	\end{aligned}\right.
    \end{equation*}
    Moreover, $u \in C([0,T];V)$ and
    \begin{equation}\label{eq:MR_estimate}
    	\norm u_{\MR_\fra(H)} \le c_0 \Big[ \norm{u_0}_V + \norm f_{L^2(0,T;H)} \Big],
    \end{equation}
    where the constant $c_0$ depends merely on $\beta_0, \beta_1, M_1, M_2, \alpha, T$, $\dot M_1$ and $\gamma$.
\end{theorem}

The fact that the solution $u$ is in $C([0,T];V)$ allows us to
relax the continuity condition on $\fra_1$ allowing a finite number
of jumps. We say that a non-autonomous form $\fra_1\colon [0,T]\times V
\times V \to \K$ is \emph{piecewise Lipschitz-continuous} if
there exist $0= t_0 < t_1 < \dots < t_n =b$ such that on each interval $(t_{i-1}, t_i)$ the form 
$\fra_1$ is the restriction of a 
Lipschitz-continuous form  on $[t_{i-1}, t_i ] \times V \times V$.  Then Theorem \ref{thm:well-posedness_in_H} remains
true.

\begin{corollary}\label{corW}
    Assume instead of {\rm d)} that $\fra_1$ is merely piecewise Lipschitz-continuous.
    Let $u_0 \in V$, $f \in L^2(0,T;H)$.
    Then there exists a unique
        $u \in \MR_\fra(H)$
    satisfying
    \begin{equation*}
    \left\{\begin{aligned}
       \dot u(t) +  B(t) \A(t)u(t)  &= f(t) \quad \text{a.e.}\\
                        u(0)    &=u_0.
	\end{aligned}\right.
    \end{equation*}
    Moreover, $u \in C([0,T];V)$.
\end{corollary}
\begin{proof}
    By Theorem \ref{thm:well-posedness_in_H} there is a solution
    $u_1 \in H^1(0,t_1;H)\cap L^2(0,t_1;V)$ on $(0,t_1)$ satisfying
    $u_1(0) = u_0$, and $u_1 \in C([0,t_1];V)$.
    Since $u_1(t_1) \in V$ we find a solution
    $u_2 \in H^1(t_1,t_2;H)\cap L^2(t_1,t_2;V) \cap C([t_1,t_2];V)$ with $u_2(t_1)=u_1(t_1)$.
    Solving successively we obtain solutions
    $u_i \in H^1(t_{i-1},t_i;H)\cap L^2(t_{i-1},t_i;V) \cap C([t_{i-1},t_i];V)$ with $u_i(t_{i-1})=u_{i-1}(t_{i-1})$ $i=1,\dots,n$.
    Letting $u(t) = u_i(t)$ for $t \in [t_{i-1},t_i)$ we obtain a solution.
    Uniqueness follows from uniqueness in Theorem \ref{thm:well-posedness_in_H}.
\end{proof}

\begin{lemma}\label{lem:diff_formula}
	\emph{(i)} Let $v \in \MR(V,V')$, then $(\A_1^{1/2})\dot{}\A_1^{-1/2}v= -\A_1^{1/2}(\A_1^{-1/2})\dot{} v$.
	
	\emph{(ii)} $\A_1^{1/2}(\A_1^{-1/2})\dot{}\,  \colon [0,T]\to\L(V,H)$ is strongly measurable and bounded.
\end{lemma}
\begin{proof}
	(i) Let $v \in \MR(V,V')$. By Theorem~\ref{thm:iso} we obtain that $\A_1^{-1/2} v \in \MR_\fra(H)$. Thus
	\begin{align*}
		\dot v &= (\A_1^{1/2}\A_1^{-1/2} v)\dot{}\\
			&=  (\A_1^{1/2})\dot{} \A_1^{-1/2} v +  \A_1^{1/2} (\A_1^{-1/2} v)\dot{} \\
			&=  (\A_1^{1/2})\dot{} \A_1^{-1/2} v +  \A_1^{1/2} (\A_1^{-1/2})\dot{} v +  \A_1^{1/2} \A_1^{-1/2} \dot v\\
			&=  (\A_1^{1/2})\dot{} \A_1^{-1/2} v +  \A_1^{1/2} (\A_1^{-1/2})\dot{} v + \dot v,
	\end{align*}
	where we have applied \eqref{eq:chain_rule_for_Au} in the second equality and \eqref{eq:chain_rule_for_v} in the third equality.
	Subtracting $\dot v$ on both sides proves the claim.
	
	(ii) $\A_1^{1/2} \colon [0,T] \to \L(V,H)$ and
		$(\A_1^{-1/2})\dot{} \colon [0,T] \to \L(V,V)$ are strongly measurable and bounded by Lemma~\ref{lem:square_root_strongly_continuous}, Lemma~\ref{lem:square_root_lipschitz_continuous} and Proposition~\ref{prop:Lip_MR}.
\end{proof}

Note that by Theorem~\ref{thm:iso} the operator $\A_1^{1/2}$ defines an isomorphism between $\MR_\fra(H)$ and $\MR(V,V')$. 
Recall also that $\A_2 \colon [0,T] \to \L(V,H)$  is strongly measurable and bounded.
\begin{proposition}~\label{prop:equiv_of_solutions}
	Let $u_0 \in V$ and $f \in L^2(0,T;H)$, then $u \in \MR_\fra(H)$ is a solution of 
	\begin{equation}\label{eq:*}
	\left\{\begin{aligned}
        \dot u(t) + B(t) \A(t)u(t)  &= f(t) \quad \text{a.e.}\\
                        u(0)    &=u_0.
        \end{aligned}\right.
    \end{equation}
    if and only if $v:= \A_1^{1/2}u \in \MR(V,V')$ is a solution of
   \begin{equation}\label{eq:**}
    \left\{\begin{aligned}
        \dot v + \A_1^{1/2}B\A_1^{1/2}v + \A_1^{1/2}B\A_2\A_1^{-1/2}v+ \A_1^{1/2}(\A_1^{-1/2})\dot{} v  &= \A_1^{1/2} f \quad \text{a.e.}\\
                        v(0)    &= \A_1^{1/2}(0)u_0.
        \end{aligned}\right.
	\end{equation}
\end{proposition}
\begin{proof}
	Let $u \in \MR_\fra(H)$ be a solution of \eqref{eq:*}, then by Theorem~\ref{thm:iso} \eqref{eq:chain_rule_for_Au} and Lemma~\ref{lem:diff_formula} (i) we obtain
	\begin{align*}
		\dot v &= \A_1^{1/2} \dot u + (\A_1^{1/2}) \dot{} u\\ 
		&= \A_1^{1/2} (f-B\A u) + (\A_1^{1/2}) \dot{} u\\
		&= \A_1^{1/2} (f-B\A \A_1^{-1/2}v) + (\A_1^{1/2}) \dot{} \A_1^{-1/2}v\\
		&= \A_1^{1/2} f -\A_1^{1/2}B\A_1^{1/2}v - \A_1^{1/2}B\A_2\A_1^{-1/2}v- \A_1^{1/2}(\A_1^{-1/2})\dot{}v.
	\end{align*}
	Hence $v$ is a solution of \eqref{eq:**}.
	
	Now suppose that $v=\A_1^{1/2}u$ satisfies \eqref{eq:**},  then by Theorem~\ref{thm:iso} \eqref{eq:chain_rule_for_v} we obtain
	\begin{align*}
		\dot u &= \A_1^{-1/2} \dot v + (\A_1^{-1/2}) \dot{} v\\ 
		&= \A_1^{-1/2} \left(\A_1^{1/2} f -\A_1^{1/2}B\A_1^{1/2}v - \A_1^{1/2}B\A_2\A_1^{-1/2}v-\A_1^{1/2} (\A_1^{-1/2})\dot{} v\right)\\ 
			&\quad+ (\A_1^{-1/2}) \dot{} v\\
		&= f -B\A_1^{1/2}v - B\A_2\A_1^{-1/2}v\\
		&= f -B\A u.
\end{align*}
	Hence $u \in \MR_\fra(H)$ is a solution of \eqref{eq:*}.
\end{proof}

\begin{proof}[Proof of Theorem~\ref{thm:well-posedness_in_H}]
	Let $u_0\in V$ and $f\in L^2(0,T;H)$.
	The operator
	\[
		\A_1^{1/2}B\A_1^{1/2} + \A_1^{1/2}B\A_2\A_1^{-1/2} + \A_1^{1/2}(\A_1^{-1/2})\dot{}
	\]
	is associated with a non-autonomous form satisfying the assumptions of Lions' theorem (Theorem~\ref{thm:Lions}),
	i.e.\ measurability, $V$-boundedness and quasi-coer\-cive\-ness.
	In fact, the first term $\langle \A_1^{1/2}(t) B(t) \A_1^{1/2}(t) u, v \rangle$ defines a measurable and $V$-bounded form by assumption h) on $B$
	and by Lemma~\ref{lem:square_root_strongly_continuous} d) and b).
	This form is coercive by assumption g)  on $B$ and b) on $\fra_1$.
	The second term defines a strongly continuous bounded mapping from $[0,T]$ into $\L(H,V')$ again by Lemma~\ref{lem:square_root_strongly_continuous}.
	And the last term is a strongly continuous bounded mapping form $[0,T]$ into $\L(V,H)$.
	Thus, by Young's inequality, the 2nd and 3rd term are perturbations preserving quasi-coerciveness.
	Hence \eqref{eq:**} has a unique solution $v \in \MR(V,V')$.
	We conclude by Proposition~\ref{prop:equiv_of_solutions} that $u:= \A_1^{-1/2}v \in \MR_\fra(H)$ is the unique solution of \eqref{eq:*}.

It remains to show the estimate \eqref{eq:MR_estimate}.
At first we note that it follows from the continuous inverse that there exists a constant $c_0>0$ such that
\begin{equation}\label{eq:mrconst}
	\norm{u}_{\MR_\fra(H)} \le c_0 \left[ \norm{u(0)}_V + \norm{\dot u + B(.)\A(.) u}_{L^2(0,T;H)} \right]
\end{equation}
for each $u \in \MR_\fra(H)$.
We have to show that this constant does not depend on the spaces and forms we choose but merely on the constants enumerated in the statement of Theorem~\ref{thm:well-posedness_in_H}.
Assume that this is false.
Then we find Hilbert spaces $V_n, H_n$ such that $V_n \underset d\hookrightarrow H_n$, $\norm{v}_{V_n} \le c_H \norm{v}_{H_n}$, we find non-autonomous forms
\[
	\fra_n \colon [0,T]\times V_n \times V_n \to \C
\]
of the form $\fra_n = \fra_{n1} + \fra_{n2}$,
as well as strongly continuous functions $B_n \colon [0,T]\to \L(H)$ such that conditions a) -- h) of Theorem~\ref{thm:well-posedness_in_H} are satisfied
with uniform constants $\beta_0, \beta_1, M_1, M_2, \alpha, T, \dot M$ and $\gamma$
and finally we find $u_n \in \MR_{\fra_n}(H_n)$ with
\[
	\norm{u_n}_{\MR_{\fra_n}(H_n)} \ge n \left[ \norm{u_n(0)}_{V_n} + \norm{\dot u_n + B_n(.)\A_n(.) u_n}_{L^2(0,T;H_n)} \right].
\]
We will show that this leads to a contradiction.
For that we consider the Hilbert products $H:=\bigoplus_{n\in \N} H_n$, $V:=\bigoplus_{n\in \N} V_n$.
Then $V \underset d\hookrightarrow H$, $\norm{v}_{V} \le c_H \norm{v}_{H}$ for all $v \in V$.
We define the form
\[
	\fra(t,u,v):= \sum_{n=1}^\infty \fra_n(t,u_n,v_n)
\]
for $u=(u_n)_{n\in\N}, v=(v_n)_{n\in\N} \in V$, $t \in [0,T]$
and the function
\[
	B\colon [0,T] \to \L(H)
\]
given by $B(t) h = (B_n(t) h_n)_{n\in\N}$ for all $h=(h_n)_{n\in\N} \in H$.
Then conditions a) -- h) of  Theorem~\ref{thm:well-posedness_in_H} are satisfied
with the same constants $\beta_0, \beta_1, M_1, M_2, \alpha, T, \dot M$ and $\gamma$ as above.
So there exists a constant $c_0>0$ such that \eqref{eq:mrconst} holds for all $u \in \MR_\fra$.
Now fix $n > c_0$, for the particular choice $u:=(0, \dots, 0, u_n, 0, \dots) \in \MR_\fra(H)$ we obtain
\begin{align*}
\norm u _{\MR_\fra(H)} &= \norm{u_n}_{\MR_{\fra_n}(H_n)}\\
	 &\ge n \left[ \norm{u_n(0)}_{V_n} + \norm{\dot u_n + B_n(.)\A_n(.) u_n}_{L^2(0,T;H_n)} \right]\\ 
	 &= n \left[ \norm{u(0)}_{V} + \norm{\dot u + B(.)\A(.) u}_{L^2(0,T;H)} \right],
\end{align*}
which is a contradiction.
\end{proof}

\begin{remark}\label{rem5.4}
    The well posedness result in Theorem \ref{thm:well-posedness_in_H} remains true if
    $\fra_1$ is merely quasi-coercive  instead of coercive.
    In fact, then we may replace $\fra_2$ by $\tilde \fra_2(t,u,v) = \fra_2(t,u,v) - \omega ({u} \mid {v})_H$
    and $\fra_1$ by $\tilde \fra_1(t,u,v) = \fra_1(t,u,v) + \omega ({u} \mid {v})_H$
    and have $\fra = \tilde \fra_1 + \tilde \fra_2$ in the desired form.
\end{remark}

\section{Applications}\label{section:applications}
This section is devoted to applications of our results on existence 
and maximal regularity of Section \ref{section:well-posedness_in_H} 
to concrete evolution equations. 
We show  how they can be applied to both
linear and non-linear evolution equations. 
We give examples illustrating the  theory without seeking for generality.
In all examples the underlying field is $\R$.

\subsection{The Laplacian with non-autonomous Robin boundary conditions}\label{subsection:example1}
Let $\Omega$ be a  bounded domain of $\R^d$ with Lipschitz boundary $\Gamma$. 
Denote by $\sigma$ the $(d-1)$-dimensional Hausdorff measure on $\Gamma$.
Let  
\[
	\beta\colon [0,T]  \times \Gamma \to \R
\]
be a bounded measurable function which is Lipschitz continuous w.r.t.\ the first variable, i.e.,
\begin{equation}\label{lipbeta}
 	\lvert \beta(t,x) - \beta(s, x) \rvert \le M \lvert t-s\rvert
\end{equation}
for some constant $M$ and all $t, s \in [0,T], \ x \in \Gamma$. We consider the symmetric form
\[
	\fra\colon [0,T] \times H^1(\Omega) \times H^1(\Omega) \to \R
\]
defined by 
\begin{equation}\label{formbeta}
	\fra(t, u, v) = \int_\Omega \nabla u \nabla v\ \d x + \int_\Gamma \beta(t, .) u v\ \d\sigma.
\end{equation}
In the second integral we omitted the trace symbol; 
we should write $u\vert_\Gamma v\vert_\Gamma$ if we want to be more precise.
The form $\fra$ is $H^1(\Omega)$-bounded and quasi-coercive. 
The first statement follows readily from the continuity 
of the trace operator and the boundedness of $\beta$. 
The second one is a consequence of the inequality 
\begin{equation}\label{trace-comp}
\int_\Gamma \lvert u \rvert^2 \ \d\sigma \le \epsilon \norm u_{H^1}^2 + c_\epsilon \norm u_{L^2(\Omega)}^2,
\end{equation}
which is valid for all $\epsilon > 0$ ($c_\epsilon$ is a constant depending on $\epsilon$). 
Note that $\eqref{trace-comp}$ is a consequence of compactness of the trace as an operator from $H^1(\Omega)$ into $L^2(\Gamma, \d \sigma)$, see \cite[Chap.\ 2 § 6, Theorem 6.2]{Nec67}.

The operator $A(t)$ associated with $\fra(t,.,.)$ on $H:= L^2(\Omega)$ 
is minus the Laplacian  with  time dependent Robin boundary conditions
\[
	\partial_\nu u(t) + \beta(t,.) u = 0 \text{ on } \Gamma.
\]
Here we use the following weak definition of the normal derivative.
Let $v \in H^1(\Omega)$ such that $\Delta v \in L^2(\Omega)$.
Let $h \in L^2(\Gamma, \d \sigma)$.  Then $\partial_\nu v = h$
by definition if
$\int_\Omega \nabla v \nabla w + \int_\Omega \Delta v w = \int_\Gamma h w \, \d \sigma$ for all $w \in H^1(\Omega)$.
Based on this definition, the domain of $A(t)$ is the set
\[
	D(A(t)) = \{ v \in H^1(\Omega) : \Delta v \in L^2(\Omega),  \partial_\nu v + \beta(t,.) v\vert_\Gamma = 0 \},
\]
and for $v\in D(A(t))$ the operator is given by $A(t)v = - \Delta v$.

By Theorem \ref{thm:well-posedness_in_H}, the heat equation 
\begin{equation*}
\left\{  \begin{aligned}
          \dot u(t)  - \Delta u(t)  &= f(t)\\
                        u(0)    &=u_0 \in H^1(\Omega)\\
                        \partial_\nu u(t) + \beta(t,.) u &= 0   \text{ on } \Gamma
                         \end{aligned} \right.
\end{equation*}
    has a unique solution  $u \in H^1(0,T;L^2(\Omega))\cap L^2(0,T;H^1(\Omega))$ whenever  $f \in L^2(0,T, L^2(\Omega))$. 
    This example is  also valid for  more general elliptic operators than the Laplacian. We could even include elliptic 
    operators with time dependent  coefficients.

\subsection{Schrödinger operators with time-dependent potentials}
Let $0 \le m_0 \in L^1_{loc}(\R^d)$ and 
$m \colon [0, T] \times \R^d \to \R$ be a measurable function for which
there exist positive constants $\alpha_1, \alpha_2 $ and $M$  such that for a.e.\ $x$
\[
	\alpha_1 m_0(x)  \le m(t, x) \le \alpha_2 m_0(x)
\]
and
\[
	\lvert m(t, x) - m(s, x) \rvert \le M \lvert t - s \rvert m_0(x) \quad x \text{-a.e.}
\]
for all $t, s \in [0,T]$.
We define the form 
\[
	\fra(t, u, v) = \int_{\R^d} \nabla u \nabla v\ \d x + \int_{\R^d} m(t,x) u v\ \d x
\]
with domain 
\[ 
	V = \Big\{ u \in H^1(\R^d) : \int_{\R^d} m_0(x) \lvert u \rvert^2 \ \d x < \infty \Big\}.
\]
It is clear that $V$ is a Hilbert space for the norm $\norm u_V$ given by 
\[
	\norm u_V^2  = \int_{\R^d} \lvert \nabla u \rvert^2 \ \d x +  \int_{\R^d} m_0(x) \lvert u \rvert^2 \ \d x.
\]
In addition, $\fra$ is $V$-bounded and coercive. 
Its associated operator on $L^2(\R^d)$ is formally given by 
\[
	A (t) = -\Delta + m(t,.).
\]
Given $f \in L^2(0,T, L^2(\R^d))$ and $u_0 \in V$, we apply Theorem \ref{thm:well-posedness_in_H} 
and obtain a unique solution $u \in H^1(0,T;L^2(\R^d))\cap L^2(0,T;V)$ of the evolution equation
 \begin{equation*}
\left\{  \begin{aligned}
          \dot u(t)  - \Delta u(t)  + m(t,.) u(t) &= f(t) \quad \text{a.e.}\\
                        u(0)    &=u_0.
                         \end{aligned} \right.
\end{equation*}

\subsection{A quasi-linear heat equation}\label{sec:quasilinear}
In this subsection we consider the non-linear evolution equation
\begin{equation*}
(\NLCP) \left\{  \begin{aligned}
          \dot u(t)  &=  m(t, x, u(t), \nabla u(t)) \Delta u(t) +  f(t)\\
                        u(0)    &=u_0 \in H^1(\Omega)\\
                        \partial_\nu u(t)&+ \beta (t, .)u (t )= 0 \text{  on  } \Gamma.\\
                         \end{aligned} \right.
                         \end{equation*}
The function $m$ is supposed to be measurable from $[0,T] \times \Omega \times \R^{1+d}$ 
with values in  $[\delta, \frac{1}{\delta}]$ for some constant  $\delta > 0$ and continuous in the last variable. 
The domain $\Omega \subset \R^d$ is bounded with Lipschitz boundary and the function $\beta$ 
satisfies $\eqref{lipbeta}$. By a \emph{solution} $u$ of $(\NLCP)$ we mean 
a function $u \in H^1(0,T, L^2(\Omega)) \cap L^2(0,T, H^1(\Omega))$ such that 
$\Delta u(t) \in L^2(\Omega)$ $t$-a.e.\ and the equality 
$$\dot u(t)  =  m(t, x, u(t), \nabla u(t)) \Delta u(t) +  f(t)$$ holds for a.e.\ $t \in [0,T]$ such that the boundary condition is satisfied (cf.\ Section~\ref{subsection:example1}).
We have the following result.
\begin{theorem}\label{thNL} 
	Let $f \in L^2(0,T, L^2(\Omega))$ and 
	$u_0 \in H^1(\Omega)$. Then  there exists a solution 
	$u \in H^1(0,T, L^2(\Omega)) \cap L^2(0,T, H^1(\Omega))$ of $(\NLCP)$.
\end{theorem}
We shall use Schauder's fixed point theorem to prove this result.  
This idea is classical in PDE but it is here that we need in an essential way 
the maximal regularity result for the corresponding 
non-autonomous linear evolution equation as well as the estimate established in 
Theorem~\ref{thm:well-posedness_in_H}.
Some of our arguments are similar to those in \cite{AC10}.  
We  emphasize that we could replace in 
$(\NLCP)$ the Laplacian by an elliptic operator with time-dependent coefficients 
(with an appropriate Lipschitz continuity with respect to $t$). 
Again, we do not search for further generality in order to make the
ideas in the proof more transparent. 

\begin{proof}[Proof of Theorem \ref{thNL}.]  
Let us denote by $H$ the Hilbert space $ L^2(\Omega)$, let $V=  H^1(\Omega)$ 
and denote by $A (t)$ the operator on $H$ associated with the form $\fra(t,.,.)$ defined by $\eqref{formbeta}$. 
As made precise in Subsection \ref{subsection:example1}, $A(t)$ is the negative Laplacian with boundary conditions 
$\partial_\nu u(t)+ \beta (t, .)u (t )= 0$  on $\Gamma$.
Given $v \in L^2(0,T, V)$ we set for $g \in H$
\[
	B_v(t) g = m(t,x,v(t), \nabla v(t)) g.
\]
Note that 
\begin{equation}\label{bound-for-B}
\delta  \norm g_H^2 \le (B_v(t)g \mid g)_H  \le \frac{1}{\delta}  \norm g_H^2.
\end{equation}
By Theorem \ref{thm:well-posedness_in_H} there exists a unique $u \in \MR(V,H) = H^1(0,T, H) \cap L^2(0,T, V)$ such that 
\begin{equation*}
\left\{  \begin{aligned}
           \dot u(t)  &= -  B_v(t) \A(t) u(t) +  f(t)\\
                        u(0)    &=u_0 \in V.
                         \end{aligned} \right.
                         \end{equation*}
Now we consider the mapping
$$S \colon L^2(0,T,V) \to L^2(0,T,V), \ Sv = u.$$
By the estimate $\eqref{eq:MR_estimate}$ of Theorem~\ref{thm:well-posedness_in_H}, we have 
\begin{equation}\label{est-MR}
\norm  u_{\MR_\fra(H)} \le C \left[ \norm f _{L^2(0,T, H)} + \norm {u_0}_V \right],
\end{equation}
with a constant $C$ which is independent of $v$. 
Since $V = H^1(\Omega)$ is compactly embedded into $H=L^2(\Omega)$ 
(recall that $\Omega$ is bounded and has Lipschitz boundary), 
we obtain from Corollary~\ref{cor:embedding_in_continuous_functions} that $\MR_\fra(H)$ is compactly embedded in $L^2(0,T;V)$.
As a consequence, it is enough to prove continuity of $S$ 
and then apply Schauder's fixed point theorem to find 
$u \in \MR(V,H)$ such that $S u = u$.  Such $u$ is a solution of $(\NLCP)$. 

Now we prove continuity of $S$. For this, we consider a sequence 
$(v_n)$ which converges to $v$ in $L^2(0,T, V)$ and let $u_n = S(v_n)$. It is enough to prove 
that  $(u_n)$ has a subsequence which converges to $S v$. 
For each $n \in \N$, $u_n$ is the solution of 
\begin{equation*} (\CP)_n  \left\{  \begin{aligned}
           \dot u_n(t)  &=   - B_{v_n}(t) \A(t) u_n(t) +  f(t)\\
                        u_n(0)    &=u_0 \in V.
        \end{aligned} \right.
\end{equation*}
By $\eqref{est-MR}$, the sequence $(u_n) $ is bounded in $\MR_\fra(H)$ and hence by extracting a subsequence we may 
assume that  $(u_n)$ converges weakly to some $u$ in $\MR_\fra(H)$. In particular $\A(.)u_n(.) \to \A(.)u$ weakly in $L^2(0,T;H)$.
Then $(u_n)_{n \in \N}$ converges in norm to $u$ in $L^2(0,T, V)$ by the compact embedding of $\MR_\fra(H)$ in $L^2(0,T, V)$. 
By extracting a subsequence again we can also assume that $v_n(t)(x) \to v(t)(x)$ and $\nabla v_n(t)(x) \to \nabla v(t)(x)$ a.e.\ with respect to $t$ and to $x$.
Hence for $g \in L^2(0,T;H)$ we have $B_{v_n}(.) g(.) \to B_v(.) g(.)$ a.e.\ and also in $L^2(0,T;H)$ by the Dominated Convergence Theorem.
Thus for all $g \in L^2(0,T;H)$ we obtain
\begin{align*}
 0 &= \int_0^T (\dot u_n(t) + B_{v_n}(t)\A(t) u_n(t) -f(t) \mid g)_H \ \d{t} \\
 	&= \int_0^T (\dot u_n(t) \mid g)_H \ \d{t} + \int_0^T (\A(t) u_n(t)\mid B_{v_n}(t) g)_H \ \d{t} - \int_0^T (f(t) \mid g)_H \ \d{t}\\
 	&\to \int_0^T (\dot u(t) \mid g)_H \ \d{t} + \int_0^T (\A(t) u(t)\mid B_{v}(t) g)_H \ \d{t} - \int_0^T (f(t) \mid g)_H \ \d{t}\\
 	&=\int_0^T (\dot u(t) + B_{v}(t)\A(t) u(t) -f(t) \mid g)_H \ \d{t} \quad (n \to \infty).
\end{align*}
Now the particular choice of $g=\dot u(t) + B_{v}(t)\A(t) u(t) -f(t)$ shows that
$\dot u(t) = - B_{v}(t)\A(t) u(t) +f(t)$.
Finally, the fact that $ \MR(V,H) \hookrightarrow C([0,T];H)$ together with the weak convergence in $\MR(V,H)$ of $(u_n)$
to $u$ imply 
\[
	u_0 = u_n(0) \to u(0).
\]
We conclude that $ u = S v$ which is the desired identity. 
\end{proof}


\section{Appendix: Mapping properties for 1-dimen\-sional Sobo\-lev spaces}

In this section we consider Sobolev spaces defined on an interval $(0,T)$,  where $T>0$,  with values in a Hilbert space $H$. Given $u\in L^2(0,T;H)$ a function
 $\dot u \in L^2(0,T;H)$ is called \emph{the weak derivative} of $u$ if
\[
	-\int_0^T  u(s) \dot \varphi (s) \ \d s = \int_0^T  \dot u (s)\varphi (s) \ \d s
\]
for all $\varphi \in C^\infty_c(0,T)$. Thus we merely test with scalar-valued test functions $\varphi$ on $(0,T)$. It is clear that the weak derivative $\dot u$ of $u$ is unique whenever it exists. We let
\[
	H^1(0,T;H):=\{u\in L^2(0,T;H):u \text{  has  a weak derivative  }  \dot u \in L^2(0,T;H)\}.
\]
It is easy to see that $H^1(0,T;H)$ is a Hilbert space for the scalar product
\[
	(u\mid v)_{H^1(0,T;H)}:=\int^T_0 \Big[ (u(t)\mid v(t))_H + (\dot u (t)\mid \dot v (t))_H \Big] \ \d t.
\]
As in the scalar case \cite[Section 8.2]{Bre11} one shows the following.

\begin{proposition}\label{prop:hauptsatz}
{\rm a)}
 Let $u\in H^1(0,T;H)$. Then there exists a unique $w\in C([0,T];H)$ such that $u(t)=w(t)$ a.e.\ and
 \[
	 w(t)=w(0)+\int^t_0 \dot u (s) \ \d s.
 \]
{\rm b)} Conversely, if $w \in C([0,T];H),v\in L^2(0,T;H)$ such that $w(t)=w(0)+\int^t_0 v(s) \, \d s$, then $w \in H^1(0,T;H)$ and $\dot w = v$.
\end{proposition}

In the following we always identify $u\in H^1(0,T;H)$ with its unique continuous representative $w$ according to a).

Next we establish a mapping theorems for Sobolev spaces.
Let $X, Y$ be Hilbert spaces.
\begin{proposition}\label{prop:lipschitz_continuous_operators}
    Let $S\colon [0,T]\to \L(Y,X)$ be Lipschitz continuous.
    Then the following holds.
    \begin{enumerate}[label={\rm \alph*)}]
        \item There exists a bounded, strongly measurable function
            $\dot S \colon [0,T] \to \L(Y,X)$ such that
            \[
                \frac \d {\d t} S(t)u = \dot S(t)u \quad (u \in Y)
            \]
            for a.e.\ $t\in [0,T]$ and
            \[
                \norm{\dot S(t)}_{\L(Y,X)} \le L \quad (t \in [0,T])
            \]
            where $L$ is the Lipschitz constant of $S$.
        \item If $ u\in H^1(0,T;Y)$, then
            $Su := S(.)u(.) \in H^1(0,T;X)$ and
            \begin{equation}\label{eq:chain_rule}
                (Su)\dot{} = \dot S(.)u(.)+ S(.) \dot u(.).
            \end{equation}
    \end{enumerate}
\end{proposition}
For the proof of Proposition~\ref{prop:lipschitz_continuous_operators} we
recall the following. If a function $u\colon[0,T]\to Y$ is absolutely
continuous, then $\dot u(t) := \frac \d{\d t} u(t)$ exists almost
everywhere and $u(t) = u(0) + \int_0^t \dot u(s) \, \d s$
\cite[Proposition 1.2.3 and Corollary 1.2.7]{ABHN11}. In fact, the
space of all absolutely continuous functions on $[0,T]$ with
values in $Y$ is the same as the Sobolev space $W^{1,1}(0,T;Y)$
and $\dot u$ coincides with the weak derivative (this is true for
a Banach space $Y$ if and only if it has the Radon-Nikodým
property). The function $u$ is in $H^1(0,T;Y)$ if and only if $u
\in W^{1,1}(0,T;Y)$ and $\dot u \in L^2(0,T;Y)$.
\begin{proof}[Proof of Proposition \ref{prop:lipschitz_continuous_operators}.]
    a) Since for $y \in Y$, $S(.)y$ is Lipschitz continuous,
        the derivative $\frac \d {\d t} S(t)y$ exists a.e.\ (see \cite[Sec. 1.2]{ABHN11}).
        Let $Y_0$ be a countable dense subset of $Y$.
        There exists a Borel null set $N \subset [0,T]$ such that
        $\frac \d {\d t} S(t)y$ exists in $X$ for all $t \notin N$
        and all $y \in Y_0$. Since $S$ is Lipschitz-continuous it
        follows easily that $ \frac \d {\d t} S(t)y$ exists also
        for all $y \in \overline{Y_0}= Y$ and $t \notin N$. Let
        \begin{equation*}
            \dot S(t) y =   \begin{cases}
                            \frac \d {\d t} S(t)y   &\text{if } t \notin N \text{ and}\\
                            \quad 0             &\text{if } t \in N.
                        \end{cases}
        \end{equation*}
        Let $L$ be the Lipschitz constant of $S$. 
	Then $\dot S(t) \in \L(Y,X)$, with $\norm{\dot S(t)}_{\L(Y,X)} \le L$ for all $t \in [0,T]$
        and $\dot S(.)y$ is measurable for all $y\in V$.

    b) Let $u \in H^1(0,T;Y)$. Then $u \in C([0,T];Y)$ and
        \[
            \norm u _\infty := \sup_{t \in [0,T]} \norm{u(t)}_Y < \infty.
        \]
        Denote the supremum norm of $S$ by
        \[
            \norm S_\infty := \sup_{t \in [0,T]} \norm{S(t)}_{\L(Y,X)}.
        \]
        We first show that $Su$ is absolutely continuous in $X$.
        Let $\epsilon > 0$. Since $u$ is absolutely continuous in $Y$
        there exists a $\delta >0$ such that
        \[
            \sum_i \norm{u(b_i)-u(a_i)}_Y \le \norm{S}_\infty^{-1} \frac \epsilon 2
        \]
        for each finite collection of non-overlapping intervals
        $(a_i, b_i)$ in $(0,T)$ satisfying $\sum_i (b_i-a_i) < \delta$.
        We may take $\delta > 0$ so small that $L \norm u_\infty \delta < \frac \epsilon 2$.
        Then
        \begin{align*}
            \sum_i \norm{S(b_i)u(b_i)& - S(a_i)u(a_i)}_X\\
                &\le \sum_i \norm{( S(b_i) - S(a_i) )u(b_i)}_X + \sum_i \norm{S(a_i)( u(b_i) - u(a_i) )}_X\\
                &\le L \sum_i (b_i-a_i) \norm u _\infty + \norm S_\infty  \norm S_\infty^{-1} \frac \epsilon 2\\
                &< L \, \delta \norm u_\infty + \frac \epsilon 2 \le \epsilon.
        \end{align*}
        Thus $Su$ is absolutely continuous. Moreover
        \begin{align*}
            (Su)\dot{}\, (t)    &= \lim_{h\to 0} \tfrac 1 h (S(t+h) u(t+h) - S(t) u(t))\\
                    &= \lim_{h\to 0} \big[ \tfrac 1 h (S(t+h)-S(t))(u(t+h)-u(t))\\
                        &\quad\quad \quad\quad + \tfrac 1 h (S(t+h)-S(t))u(t)\\
                        &\quad\quad \quad\quad + \tfrac 1 h S(t)(u(t+h)-u(t)) \big]\\
                    &= \dot S(t) u(t) + S(t) \dot u(t) \quad \text{a.e.}
        \end{align*}
        Thus $(Su)\dot{} \in L^2(0,T;X)$ and so $Su\in H^1(0,T;X)$.
\end{proof}

Next we consider a mapping theorem for $\MR$-spaces.
Let $X, Y, Z$ be Hilbert spaces such that $Y \underset d \hookrightarrow X$. We let 
$ \MR(Y,X) := H^1(0,T; X) \cap L^2(0,T;Y)$. 

\begin{proposition}\label{prop:Lip_MR}
	Let $S\colon[0,T]\to \L(X,Z)$ be strongly measurable and bounded such that $S$ is Lipschitz continuous with values in $\L(Y,Z)$; i.e.,
	\[
		\norm{S(t)y-S(s)y}_Z \le L \abs{t-s} \norm{y}_Y \quad (y \in Y,\ t,s \in [0,T]).
	\]
	Then $Su \in H^1(0,T;Z)$ for all $u \in \MR(Y,X)$ and
	\[
		(Su)\dot{} = \dot S u + S \dot u.
	\]
\end{proposition}
\begin{proof}
	1) Let $u \in H^1(0,T;Y)$. Then by Proposition~\ref{prop:lipschitz_continuous_operators} we have $Su \in H^1(0,T;Z)$ and $(Su)\dot{} = \dot S u + S \dot u$.
	By the assumptions on $S$ there exists a constant $c\ge 0$ such that
	\[
		\norm{Su}_{H^1(0,T;Z)} \le c \norm{u}_{\MR(Y,X)}
	\]
	for all $u \in H^1(0,T;Y)$.
	
	2) Let $u \in \MR(Y,X)$. By \cite[p.\ 105]{Sho97} there exists $u_n \in H^1(0,T;Y)$ such that $u_n \to u$ in $\MR(Y,X)$ as $n \to \infty$.
	It follows that $Su_n \to Su$ and $\dot S u_n + S \dot u_n \to \dot S u + S \dot u$ in $L^2(0,T;Z)$. Thus for $\varphi \in \D(0,T)$,
	\begin{align*}
		- \int_0^T Su \dot \varphi \ \d{t} &= \lim_{n\to\infty} - \int_0^T Su_n \dot \varphi \ \d{t}\\
			&= \lim_{n\to\infty} \int_0^T (\dot S u_n + S \dot u_n)  \varphi \ \d{t}\\
			&= \int_0^T (\dot S u + S \dot u)  \varphi \ \d{t}.
	\end{align*}
	This proves the claim.
\end{proof}
Note that  Proposition \ref{prop:lipschitz_continuous_operators} remains true if $X$ and $Y$ are Banach spaces such that $Y$ 
has the Radon-Nikodým property (e.g., if $Y$ is reflexive or a separable dual space). Proposition \ref{prop:Lip_MR} remains true if $Z$ has the 
Radon-Nikodým property. Moreover, in all these cases $H^1$ could be replaced by $W^{1,p}, 1 \le p < \infty$. 


\noindent
\emph{Wolfgang Arendt}, \emph{Dominik Dier}, Institute of Applied Analysis, 
University of Ulm, 89069 Ulm, Germany,\\
\texttt{wolfgang.arendt@uni-ulm.de}, \texttt{dominik.dier@uni-ulm.de}

\quad\\
\noindent
\emph{Hafida Laasri}, Fachbereich C - Mathematik und Naturwissenschaften, University of Wuppertal, Gaußstraße 20,
42097 Wuppertal, Germany,\\
\texttt{laasrihafida@gmail.com}

\quad\\
\noindent
\emph{El Maati Ouhabaz,} Institut de Math\'ematiques (IMB), Univ.\  Bordeaux, 351, cours de la Libération, 33405 Talence cedex, France,\\ 
\texttt{Elmaati.Ouhabaz@math.u-bordeaux1.fr}

\end{document}